\documentclass[a4paper,11pt,reqno]{amsart}

\usepackage[colorlinks,
           linkcolor=blue,      
           anchorcolor=blue,  
           citecolor=red,       
            ]{hyperref}
\usepackage{amsmath,amssymb,amsthm}
\usepackage{latexsym}
\usepackage{color}
\usepackage{graphicx}
\usepackage{mathrsfs}
\usepackage{enumerate}
\usepackage[abbrev]{amsrefs}
\usepackage[T1]{fontenc}
\usepackage{mathtools}
\mathtoolsset{showonlyrefs=true}
\allowdisplaybreaks[4]

\theoremstyle{plain}
\newtheorem{thm}{Theorem}[section]
\newtheorem{lemm}[thm]{Lemma}

\theoremstyle{definition}

\newtheorem{rem}[thm]{Remark}

\makeatletter

\@addtoreset{equation}{section}
\makeatother

\newcommand{\N}{\mathbb{N}}
\newcommand{\Z}{\mathbb{Z}}

\newcommand{\R}{\mathbb{R}}

\newcommand{\dB}{\dot{B}}

\newcommand{\dH}{\dot{H}}

\newcommand{\supp}{\operatorname{supp}}

\renewcommand{\leq}{\leqslant}
\renewcommand{\geq}{\geqslant}

\newcommand{\Om}{\Omega}

\renewcommand{\div}{\operatorname{div}}

\newcommand{\n}[1]{{\left\|#1\right\|}}
\newcommand{\abso}[1]{{\left|#1\right|}}
\newcommand{\lp}[1]{\left[#1\right]}
\newcommand{\Mp}[1]{\left\{#1\right\}}
\renewcommand{\sp}[1]{\left(#1\right)}

\newcommand{\dtau}{ \, {{\rm d}} \tau  }

\begin{document}
\title[Decay of rotating MHD]
{
Refined decay estimates for global solutions to the rotating MHD equations
}
\author[M.~Fujii]{Mikihiro Fujii}
\address[M.~Fujii]{Graduate School of Science, Nagoya City University}
\email[M.~Fujii]{fujii.mikihiro@nsc.nagoya-cu.ac.jp}
\author[Y.~Li]{Yang Li}
\address[Y.~Li]{School of Mathematical Sciences, Anhui University, Hefei, 230601, People's Republic of China}
\email[Y.~Li]{lynjum@163.com}
\begin{abstract}
In this paper, we consider the Cauchy problem for the incompressible magnetohydrodynamic equations with the Coriolis force in the three-dimensional whole space. For large initial data, it is known that the Cauchy problem admits a unique global solution in critical Sobolev space $\dot{H}^{1/2}(\R^3)$ provided that the speed of rotation is fast enough. By using delicate dispersive estimates, we show refined decay estimates of the velocity field and magnetic field. These decay rates significantly improve the ones obtained by Kim [\emph{J. Differential Equations.}, 2022] in the subcritical Sobolev framework $H^s(\R^3)$ with $1/2<s<3/2$.  
\end{abstract}
\maketitle

\tableofcontents

\section{Introduction}
In this paper, we consider the Cauchy problem for the incompressible magnetohydrodynamic (`MHD' in short) equations with the Coriolis force in $\R^3$. In Eulerian coordinates, the governing equations take the form: 
\begin{align}\label{eq:non}
    \begin{cases}
        \partial_t u - \Delta u + \Omega e_3 \times u + (u \cdot \nabla)u-
        (B \cdot\nabla) B+ \nabla P =0, & t>0,x \in \R^3, \\
        \partial_t B-\Delta B+ (u \cdot \nabla)B-(B\cdot \nabla)u=0,  & t>0,x \in \R^3, \\
        \div u =\div B= 0, & t\geq 0, x \in \R^3, \\
        u(0,x) = u_0(x),\quad B(0,x)=B_0(x), & x \in \R^3.
    \end{cases}
\end{align}
Here, $u=u(t,x):[0,\infty)\times \R^3 \rightarrow \R^3$ and $B=B(t,x):[0,\infty)\times \R^3 \rightarrow \R^3$ denote the unknown velocity field and magnetic field respectively. $P=P(t,x):(0,\infty) \times \R^3 \rightarrow \R$ denotes the scalar pressure. The rotating axis is $e_3:=(0,0,1)$ and the speed of rotation is denoted by $\Omega\in \R$.

Now we briefly recall some previous results related to \eqref{eq:non}. If $\Omega=0$, \eqref{eq:non} reduces to the classical incompressible MHD equations. Mathematical results on this system is nowadays classical, see \cite{ST83}. For large initial data, the system admits global regular solution in $2$D and global existence of weak solution in $3$D. For small initial data, the system admits global small smooth solution in $3$D. For the decay estimates of global solutions in $L^2$, we refer to \cites{Ko89,MS1,SSS-96}.  

When $B=0$, \eqref{eq:non} reduces to the incompressible rotating Navier--Stokes equations. Mathematical studies of the incompressible rotating Navier--Stokes equations originates from a series of work by Babin et al. \cites{BMN97,BMN99,BMN01}. We refer to \cites{CDGG3,CDGG2,CDGG,Fuj-24,HS10,IT13,IT14,KLT14} for more results on global well-posedness of the incompressible rotating Navier--Stokes equations. Later, it was found that fast rotation, acting as the dispersive effect, enhances the decay estimates of velocity field, see \cites{Ahn-Kim-Lee-22,ET23,Yos-26}; very recently, Fujii et al. \cite{FLX26} improved the decay rates by delicate dispersive estimates.

Now we focus on the full system \eqref{eq:non}. Due to the Coriolis force and the mutual interactions between the velocity field and magnetic field, the analysis for \eqref{eq:non} is more involved. For small initial data, Abidin and Chen \cite{AC21} proved the global well-posedness of \eqref{eq:non} in Fourier-Besov spaces. Ahn et al. \cite{AKL21} proved the global well-posedness for large initial data, when $|\Omega|$ is sufficiently large. More precisely, Ahn et al. \cite{AKL21} assumed that
\begin{align}
  &  u_0 \in H^s(\R^3), \quad B_0 \in (L^2 \cap L^q)(\R^3), \\
  & \frac{1}{2}<s <\frac{3}{4}, \quad 
  3<q < \min\left\{ 
  \frac{6}{3-2s}, \frac{27}{6+4s}
  \right\}. 
\end{align}
This result was later improved by Kim \cite{Kim22}, who proved the global well-posedness for \eqref{eq:non} when 
\begin{align}
    & (u_0,B_0) \in H^{s}(\R^3), \quad \frac{1}{2}<s <\frac{3}{2},
    \quad 0<\gamma \leq s-\frac{1}{2},
\end{align}
and $|\Omega|$ is sufficiently large compared with $\n{(u_0,B_0)}_{ H^{\gamma+ 1/2} }$. Assuming moreover that $(u_0,B_0)\in L^1(\R^3)$ and 
\begin{align}\label{decay-Kim-q}
 \frac{1}{4}+\frac{s}{6}< \frac{1}{q}\leq \frac{1}{2}, \quad 
 \frac{1}{q}> \frac{4-\gamma}{8+\gamma},
\end{align}
Kim \cite{Kim22} further obtained the decay estimates as 
\begin{align}
 &   \n{u(t)}_{L^q} \leq C t^{-\frac{3}{2}(1-\frac{1}{q})  }
 (|\Omega|t)^{-(1-\frac{2}{q}) }, \label{decay-kim}  \\
 & \n{B(t)}_{L^q} \leq C t^{-\frac{3}{2}(1-\frac{1}{q})  } ,  \label{decay-kim-B}
\end{align}
when $|\Omega|$ is sufficiently large compared with $\n{(u_0,B_0)}_{ \dH^{\gamma+ 1/2} \cap L^1 }$. Recently, Takada and Yoneda \cite{Tak-Yon-24} proved the global well-posedness of \eqref{eq:non} in critical Sobolev space $\dot{H}^{1/2}(\R^3)$ with large initial data, as long as $|\Om|$ is sufficiently large.

Our aim of this paper is to improve the decay estimates in Kim \cite{Kim22}. More precisely, 
\begin{itemize}
\item{
Under the divergence-free condition and $L^1$ assumption of $(u_0,B_0)$, it is natural to expect the slightly faster decay estimates, for instance instead of \eqref{decay-kim-B}, 
\begin{align}
     \n{B(t)}_{L^q} =o\sp{  t^{-\frac{3}{2}(1-\frac{1}{q})  } } \quad \text{as}\quad t \rightarrow \infty.  
\end{align}
}
    \item{
    From \eqref{decay-kim} we see that the fast rotation gives the additional decay rate $-(1-2/q)$ for the velocity field. Inspired by our recent work \cite{FLX26}, we aim to obtain a \emph{faster} decay rate than $-(1-2/q)$. 
    }
    \item{
    Seeing that \eqref{decay-kim} and \eqref{decay-kim-B} hold for $2\leq q<3$ satisfying \eqref{decay-Kim-q}, we aim to obtain the decay estimates of the velocity field and magnetic field for \emph{general} $q$ ranging in $[2,\infty]$.  
    }
\end{itemize}

Now we are ready to state the main results of this paper. The first theorem treats the $L^p$ decay estimates of global solutions to \eqref{eq:non} with critical Sobolev regularity. 
\begin{thm}\label{thm-1}
    Let $2\leq p\leq \infty$ and $\ell \in \N \cup \{0\}$.
    Let $u_0,B_0 \in \dot H^{\frac{1}{2}}(\R^3)\cap L^1(\R^3)$ satisfy $\div u_0 =\div B_0= 0$.
    Then, there exists a positive constant $\Omega_0=\Omega_0(u_0,B_0)$ such that
    for any $\Omega \in \R$ with $|\Omega| \geq \Omega_0$, \eqref{eq:non} possesses a unique global solution $u,B \in C([0,\infty);H^{\frac{1}{2}}(\R^3)) \cap C((0,\infty);L^\infty(\R^3))$. Moreover, we have the decay estimates 
    \begin{align}\label{non-dec-L^p}
      &  \n{\nabla^{\ell} u(t)}_{L^p} = 
        \begin{cases}
            o\sp{t^{-\frac{3}{2}(1-\frac{1}{p})-\frac{\ell}{2}}(1+|\Omega|t)^{-\frac{3}{2}(1-\frac{2}{p})}} & (2 \leq p <4),\\
            o\sp{t^{-\frac{3}{2}(1-\frac{1}{4})-\frac{\ell}{2}}(1+|\Omega|t)^{-\frac{3}{4}}\sp{\log(e+|\Omega|t)}^{\frac{1}{4}}} & (p = 4),\\
            o\sp{t^{-\frac{3}{2}(1-\frac{1}{p})-\frac{\ell}{2}}(1+|\Omega|t)^{-(1-\frac{1}{p})}} & (4<p\leq \infty) , 
        \end{cases} \\
        & \n{\nabla^{\ell} B(t)}_{L^p} = o\sp{   t^{-\frac{3}{2}(1-\frac{1}{p})-\frac{\ell}{2}}  } 
    \end{align}
    as $t \to \infty$.
\end{thm}
\begin{rem}
It was proved by Kim \cite{Kim22} that the global solution to \eqref{eq:non}, see \eqref{decay-kim}-\eqref{decay-kim-B}, decays as
\begin{align}
    \n{u(t)}_{L^q}=O \sp{ t^{-\frac{3}{2}(1-\frac{1}{q})-(1-\frac{2}{q}) }  }, \quad
      \n{B(t)}_{L^q}=O \sp{ t^{-\frac{3}{2}(1-\frac{1}{q}) }  }, \quad t\rightarrow \infty
\end{align}
for suitable $2\leq q <3$. Theorem \ref{thm-1} gives an improved decay estimate: 
\begin{align}\label{non-dec-L^p-rem}
      &  \n{ u(t)}_{L^p} = 
        \begin{cases}
            o\sp{t^{-\frac{3}{2}(1-\frac{1}{p})}
           t^{-\frac{3}{2}(1-\frac{2}{p})}} & (2 \leq p <4), \\
            o\sp{t^{-\frac{3}{2}(1-\frac{1}{4})}
            t^{-\frac{3}{4}}\sp{\log t}^{\frac{1}{4}}} & (p = 4), \quad t\rightarrow \infty \\
            o\sp{t^{-\frac{3}{2}(1-\frac{1}{p})}
            t^{-(1-\frac{1}{p})}} & (4<p\leq \infty) , \quad 
        \end{cases} \\
        & \n{B(t)}_{L^p} = o\sp{   t^{-\frac{3}{2}(1-\frac{1}{p})}  } ,
        \quad t\rightarrow \infty
    \end{align}
    for all $2\leq p \leq \infty$. Hence, we improve the decay rates obtained by Kim \cite{Kim22}. 
\end{rem}

Similar to the incompressible (rotating) Navier--Stokes equations \cites{ET23,Fuj-Miya-01,FLX26} and incompressible MHD equations \cite{L24}, the decay rates obtained above may be improved to the additional $t^{-\frac{1}{2}}$-faster decay estimates, provided that $|x|u_0(x)$ and $|x|B_0(x)$ are integrable.
\begin{thm}\label{thm-2}
    In addition to the hypotheses of Theorem \ref{thm-1}, assume moreover that $|x|u_0(x), |x|B_0(x) \in L^1(\R^3_x)$.
    Then, for every $2 \leq p \leq \infty$ and $\ell \in \N \cup \{0\}$, 
    there exists a positive constant $C=C(\ell,p,u_0,B_0)$ such that the global solution $(u,B)$ to \eqref{eq:non} with $|\Omega| \geq \Omega_0$, constructed in Theorem \ref{thm-1}, satisfies the $1/2$-enhanced decay estimates
    \begin{align}
        &
        \n{\nabla^{\ell} u(t)}_{L^p}
        \leq
        \begin{cases}
            Ct^{-\frac{3}{2}(1-\frac{1}{p})-\frac{\ell+1}{2}}(1+|\Omega|t)^{-\frac{3}{2}(1-\frac{2}{p})} & (2 \leq p <4),\\
            Ct^{-\frac{3}{2}(1-\frac{1}{4})-\frac{\ell+1}{2}}(1+|\Omega|t)^{-\frac{3}{4}}\sp{\log(e+|\Omega|t)}^{\frac{1}{4}} & (p = 4), \label{enh-non-dec-L^p}  \\
            Ct^{-\frac{3}{2}(1-\frac{1}{p})-\frac{\ell+1}{2}}(1+|\Omega|t)^{-(1-\frac{1}{p})} & (4<p\leq \infty),  
        \end{cases} \\
        & 
        \n{\nabla^{\ell} B(t)}_{L^p} \leq C   t^{-\frac{3}{2}(1-\frac{1}{p})-\frac{\ell+1}{2}}  \label{non-dec-B}
    \end{align}
    for all $t \geq 1$. For $p=\infty$, the constant $C$ in \eqref{enh-non-dec-L^p} and \eqref{non-dec-B} additionally depends on $\Omega$. Furthermore, the logarithmic correction for $p=4$ may be removed in $L^{4,\infty}$ framework:
    \begin{align}\label{corrc-4}
        \n{\nabla^{\ell} u(t)}_{L^{4,\infty}}
        \leq 
        C
        t^{-\frac{3}{2}(1-\frac{1}{4})-\frac{\ell+1}{2}}
        (1+|\Omega|t)^{-\frac{3}{4}}
    \end{align}
    for all $t \geq 1$.
\end{thm}

\section{Preliminaries}

To begin with, we recall some function spaces that will be used in the sequel. 
We denote by $\mathscr{S}(\R^3)$ the space of all Schwartz functions on $\R^3$ and $\mathscr{S}'(\R^3)$ the set of all tempered distributions on $\R^3$. 
For any $f\in \mathscr{S}(\R^3)$, we define the Fourier transform and inverse Fourier transform by
\begin{align}
\mathscr{F}[f](\xi)=\widehat{f}(\xi)=
\int_{\R^3} e^{- i x\cdot \xi} f(x)\,dx, \qquad
\mathscr{F}^{-1}[f](x)= (2\pi)^{-3} \int_{\R^3} e^{ i x\cdot \xi} f(\xi)\,d\xi.
\end{align} 
Let $\psi \in C_c^{\infty}([0,\infty);[0,1])$  satisfy
$\supp \psi \subset [2^{-1},2]$ and $\sum_{j \in \Z} \psi(2^{-j}r) = 1$ for all $r>0$.
The dyadic localization operators $\{\Delta_j\}_{j \in \Z}$ is defined by
\begin{align}
    \Delta_j f := \mathscr{F}^{-1}\lp{\widehat\varphi_j(\xi)\widehat{f}(\xi)},
    \qquad
    \widehat {\varphi_j}(\xi):=\psi(2^{-j}|\xi|).
\end{align}
For any $s\in \R,1\leq p,\sigma\leq \infty$, we define the homogeneous Besov space $\dB^{s}_{p,\sigma}(\R^3)$ by 
\begin{align}
 \dB^{s}_{p,\sigma}(\R^3)&
 := \left\{  f \in \mathscr{S}'(\R^3)/\mathscr{P}(\R^3) ;\,
\n{f}_{  \dB^{s}_{p,\sigma} } <\infty 
\right\}, \\
\n{f}_{  \dB^{s}_{p,\sigma} } & :=
\n{
\Mp{ 2^{sj} \n{\Delta_j f}_{L^p} }_{j \in \Z}
}_{\ell^{\sigma}(\Z)}, 
\end{align} 
where $\mathscr{P}(\R^3)$ denotes the set of all polynomials on $\R^3$. When $p=\sigma=2$, we identify $\dB^{s}_{2,2}(\R^3)$ with the homogeneous Sobolev space $\dH^{s}(\R^3)$ in the sense of norm equivalence. We refer to \cite{Bah-Che-Dan-11} for more details on Besov spaces. 
Finally, we recall the definition of weak Lebesgue spaces. For $1\leq q< \infty$, the weak Lebesgue space $L^{q,\infty}(\R^3)$ consists of all measurable functions $f$ on $\R^3$ such that
\begin{align}
    \n{f}_{L^{q,\infty}}
    :=
    \sup_{\lambda>0} 
    \lambda
    \abso{\{x \in \R^3;\, |f(x)|>\lambda  \}}^\frac{1}{q}
    <\infty, 
\end{align} 
It is well-known that $L^q(\R^3)\subsetneq L^{q,\infty}(\R^3)$ when $1\leq q<\infty$.

Next, we recall a real interpolation inequality for Besov spaces.
\begin{lemm}[\cite{Bah-Che-Dan-11}*{Proposition 2.22}]\label{lemm:interp}
    Let $1\leq p \leq \infty$ and $s_0,s_1 \in \R$ with $s_0<s_1$. Let $0<\theta<1$. Then there exists a positive constant $C=C(s_1,s_2,\theta)$ such that
    \begin{align}
        \n{f}_{\dB_{p,1}^{\theta s_1 + (1-\theta)s_0}} 
        \leq 
        C
        \n{f}_{\dB_{p,\infty}^{s_1}}^{\theta}
        \n{f}_{\dB_{p,\infty}^{s_0}}^{1-\theta} 
    \end{align}
    for all $f \in \dB_{p,\infty}^{s_1}(\R^3) \cap \dB_{p,\infty}^{s_2}(\R^3)$. 
\end{lemm}

We also recall the basic para-product estimate in Besov spaces. 
\begin{lemm}[\cite{FLX26}*{Lemma 2.3}]\label{lemm:prod-est}
    For $s>0$, $1\leq p,p_1,p_2,\sigma\leq \infty$ with $1/p  = 1/p_1 + 1/p_2$, there exists a positive constant $C=C(s,p,p_1,p_2,\sigma)$ such that 
    \begin{align}
        \n{fg}_{ \dB^{s}_{p,\sigma}}
        \leq 
        C \sp{
        \n{f}_{L^{p_1}} 
        \n{g}_{\dB^{s}_{p_2,\sigma}} 
        + 
        \n{f}_{\dB_{p_2,\sigma}^s}\n{g}_{L^{p_1}} 
        }
    \end{align} 
    for all  $f,g\in L^{p_1}(\R^3) \cap \dB^{s}_{p_2,\sigma}(\R^3)$.
\end{lemm}

Next, we recall the global well-posedness of \eqref{eq:non} in critical Sobolev space $\dot{H}^{1/2}(\R^3)$ with large initial data, due to Takada and Yoneda \cite{Tak-Yon-24}. 
\begin{lemm}\label{lemm:GWP}
Assume that $u_0,B_0\in \dot{H}^{1/2}(\R^3)$ with $\div u_0=\div B_0=0$. Then, there exists a positive constant $\Omega_0=\Omega_0(u_0,B_0)$ such that \eqref{eq:non} admits a unique global solution $(u,B)$ in the class
\begin{align}
    u,B \in C([0,\infty); \dot{H}^{\frac{1}{2}}(\R^3) )
    \cap L^2(0,\infty; \dot{H}^{\frac{3}{2}}(\R^3) )
\end{align}
for all $\Omega \in \R $ satisfying $|\Omega|\geq \Omega_0$. 
\end{lemm}

To continue, we recall the decay estimates of global solutions to \eqref{eq:non} in $L^2$. 
\begin{lemm}\label{lemm:L2-decay}
In addition to the assumptions of Lemma \ref{lemm:GWP}, assume moreover that $(u_0,B_0)\in L^1(\R^3)$. Then the global solution obtained by Lemma \ref{lemm:GWP} belongs to
\begin{align}
    u,B \in C([0,\infty); H^{\frac{1}{2}}(\R^3) )
    , \quad 
     \nabla u, \nabla B \in  L^2(0,\infty; H^{\frac{1}{2}}(\R^3) )
\end{align}
for all $\Omega \in \R $ satisfying $|\Omega|\geq \Omega_0$. Moreover, for any $\ell \in \mathbb{N}$, there exists a positive constant $C_{\ell}$, depending only on $\n{u_0}_{L^1},\n{B_0}_{L^1},\n{u_0}_{\dot{H}^{1/2}}, \n{B_0}_{\dot{H}^{1/2}} $, such that
\begin{align}
\n{\nabla^{\ell} u (t)}_{L^2} \leq C_{\ell} (1+t)^{-\frac{3}{4}} t^{-\frac{\ell}{2} }, \quad 
\n{\nabla^{\ell} B (t)}_{L^2} \leq C_{\ell} (1+t)^{-\frac{3}{4}} t^{-\frac{\ell}{2} }
\end{align}
for all $t>0$. 
\end{lemm}
We remark the the $L^2$ decay estimates in Lemma \ref{lemm:L2-decay} follows from the classical result on decay estimates of the standard incompressible MHD equations and Fourier's splitting method, see \cites{Ko89,MS1,SSS-96}. Notice that the rotational effect disappears in $L^2$ setting. A detailed proof is thus omitted.

Finally, we recall an elementary estimate from \cite{FLX26}*{Lemma 4.3}. 
\begin{lemm}\label{lemm:elem}
Let $\mathcal{D}_p(\cdot)$ be defined by \eqref{def-D}. For $2 \leq p \leq \infty$, there exists a positive constant $C=C(p)$ such that 
    \begin{align}
        \int_0^t \mathcal{D}_p(|\Om|s)\, ds \leq 
        \begin{cases}
            Ct\mathcal{D}_p(|\Om| t) & (2 \leq p<\infty), \\
            C|\Om|^{-1}\log (e + |\Om| t) & (p=\infty)
        \end{cases}
    \end{align}
    for all $t > 0$ and $\Om \in \R \setminus \{0\}$.
\end{lemm}

\section{Linear analysis}
In this section, we focus on the linearized equations: 
\begin{align}\label{eq:lin}
    \begin{cases}
        \partial_t u^{\rm lin} - \Delta u^{\rm lin} + \Omega e_3 \times u^{\rm lin} + \nabla P =0, & t>0,x \in \R^3, \\
        \partial_t B^{\rm lin}-\Delta B^{\rm lin}=0,  & t>0,x \in \R^3, \\
        \div u^{\rm lin} =\div B^{\rm lin}= 0, & t\geq 0, x \in \R^3, \\
        u^{\rm lin}(0,x) = u_0(x),\quad B^{\rm lin}(0,x)=B_0(x) & x \in \R^3.
    \end{cases}
\end{align}
Following \cites{ET23,HS10}, we know that the solution to \eqref{eq:lin} takes the form
\begin{align}\label{lin-sol-for}
    \begin{cases}
       u^{\rm lin}(t)= e^{t (\Delta-\Om \mathbb{P} e_3 \times \mathbb{P})  } u_0
=
e^{t \Delta} e^{it \Om \frac{D_3}{|D|}} P_{+}(D) u_0
+
e^{t \Delta} e^{-it \Om \frac{D_3}{|D|}} P_{-}(D) u_0, \\
       B^{\rm lin}(t)= e^{t \Delta}B_0,
    \end{cases}
\end{align}
where the Fourier multiplier $P_\pm(D)=\mathscr{F}^{-1} P_\pm(\xi)\mathscr{F}$ is defined by 
\begin{align}\label{def-of-PR}
    P_\pm(\xi) = \frac{1}{2}\sp{I \pm iR(\xi)},
    \qquad
    R(\xi)
    :=
    \frac{1}{|\xi|}
    \begin{pmatrix}
        0 &  -  \xi_3 & \xi_2 \\
        \xi_3 & 0 & -  \xi_1 \\
        - \xi_2 & \xi_1 & 0
    \end{pmatrix}. 
\end{align}
For brevity, we shall denote by $T_{\Om}(t):=u^{\rm lin}(t)$ in the sequel and introduce the notation 
\begin{align}\label{def-D}
    \mathcal{D}_p(\tau)
    :=
    \begin{cases}
        {(1+|\tau|)^{-\frac{3}{2}(1-\frac{2}{p})}}  & {\rm for}\quad 2 \leq p < 4,
        \\
        {(1+|\tau|)^{-\frac{3}{4}}(\log(e+|\tau|))^{\frac{1}{4}}}  & {\rm for}\quad p=4, 
        \\
        {(1+|\tau|)^{-(1-\frac{1}{p})}}  & {\rm for}\quad 4<p \leq \infty.
    \end{cases}
\end{align}
Here, $2 \leq p \leq \infty$ and $\tau \in \R$.

Now, we recall some previous decay estimates related to the linear solution \eqref{lin-sol-for}.  
\begin{lemm}\label{thm:lin-decay}
    For $2 \leq p \leq \infty$, $\ell \in \N \cup \{0\}$,
    the following statements hold true.
    \begin{itemize}
    \item [(1)]
    For any $u_0 \in L^1(\R^3)$ with $\div u_0=0$,
    it holds
    \begin{align}
        \n{\nabla^{\ell} T_\Omega(t)u_0}_{L^p}
        =
        o\sp{t^{-\frac{3}{2}(1-\frac{1}{p})-\frac{\ell}{2}}
            \mathcal{D}_p(|\Omega| t)}
    \end{align}
    as $t \to \infty$.
    \item [(2)]
    If $|x|u_0(x) \in L^1(\R^3_x)$ with $\div u_0=0$, then there exists a positive constant $C=C(p,\ell)$ such that
    \begin{align}
        \n{\nabla^{\ell} T_\Omega(t)u_0}_{L^p}
        \leq 
        Ct^{-\frac{3}{2}(1-\frac{1}{p})-\frac{\ell+1}{2}}
        \mathcal{D}_p(|\Omega| t)\n{|x|u_0(x)}_{L^1(\R^3_x)}
    \end{align}
    for all $t >0$.
    \end{itemize}
\end{lemm}
\begin{rem}\label{Rem}
Indeed, from the proof of Lemma \ref{thm:lin-decay} we see that there exists a positive constant $C$ such that 
    \begin{align}
        \n{\Delta_j 
        { e^{i\tau \frac{D_3}{|D|}}f }  }_{L^p}
        &
        \leq 
        C
        2^{3(1-\frac{1}{p})j}
        \mathcal{D}_p(\tau)
        \n{\Delta_j f}_{L^1},
        \\
        \n{\Delta_j 
        { e^{i\tau \frac{D_3}{|D|}}f}  }_{L^{4,\infty}}
        &
        \leq 
        C
        2^{3(1-\frac{1}{4})j}
        (1+|\tau|)^{-\frac{3}{4}}
        \n{\Delta_j f}_{L^1}
    \end{align}
    for all $\tau \in \R$, $j \in \Z$ and $f \in \mathscr{S}'(\R^3)$ with $\Delta_j f \in L^1(\R^3)$.  
\end{rem}
The proof of Lemma \ref{thm:lin-decay} can be found in \cite{FLX26}*{Theorem 3.2}. The details are omitted for brevity.

\begin{lemm}\label{thm:heat-decay}
 For $1 \leq p \leq \infty$, $\ell \in \N \cup \{0\}$,
    the following statements hold true.
    \begin{itemize}
    \item [(1)]
    For any $B_0 \in L^1(\R^3)$ with $\div B_0=0$,
    it holds
    \begin{align}
        \n{\nabla^{\ell} e^{t \Delta}B_0}_{L^p}
        =
        o\sp{t^{-\frac{3}{2}(1-\frac{1}{p})-\frac{\ell}{2}}
          }
    \end{align}
    as $t \to \infty$.
    \item [(2)]
    If $|x|B_0(x) \in L^1(\R^3_x)$ with $\div B_0=0$, then there exists a positive constant $C=C(p,\ell)$ such that
    \begin{align}
        \n{\nabla^{\ell}e^{t \Delta}B_0}_{L^p}
        \leq 
        Ct^{-\frac{3}{2}(1-\frac{1}{p})-\frac{\ell+1}{2}}
        \n{|x| B_0(x)}_{L^1(\R^3_x)}
    \end{align}
    for all $t >0$.
    \end{itemize}

\end{lemm}
The proof of Lemma \ref{thm:heat-decay} may be found in \cite{L24}*{Lemmas 2.1, 2.5}. The details of proof are omitted.

\section{Nonlinear analysis}
\subsection{Decay estimates of the velocity field}
By Duhamel's principle, we rewrite the solution to \eqref{eq:non} as
\begin{align}\label{Duham-non}
    \begin{cases}
      u(t)= T_{\Omega} u_0-  \mathcal{N}_{\Om}^{\rm vel}[u](t) + \mathcal{N}_{\Om}^{\rm vel}[B](t)  , \\
      B(t)=e^{t \Delta}B_0  -  \mathcal{N}_{\Om}^{\rm mag}[u,B](t)     ,
    \end{cases}
\end{align}
where 
\begin{align}
& \mathcal{N}_{\Om}^{\rm vel}[u](t):= 
\int_0^t T_{\Omega} (t-\tau) \mathbb{P} \div (u \otimes u)(\tau) \dtau, \\
& \mathcal{N}_{\Om}^{\rm vel}[B](t):= 
\int_0^t T_{\Omega} (t-\tau) \mathbb{P} \div (B \otimes B)(\tau) \dtau, \\
& \mathcal{N}^{\rm mag}[u,B](t) := \int_0^t e^{(t-\tau) \Delta} \nabla \times (B \times u) (\tau) \dtau,
\end{align}
and $T_{\Omega}(\cdot)$ was defined in \eqref{lin-sol-for}. Notice that the decay estimate for $\mathcal{N}_{\Om}^{\rm vel}[u]$ has been obtained in our previous work \cite{FLX26}*{Proposition 4.2}. This is summarized as follows.
\begin{lemm}\label{prop:nonl-dec}
    Let $u_0,B_0 \in \dot H^{\frac{1}{2}}(\R^3) \cap L^1(\R^3)$ with $\div u_0=\div B_0 = 0$,
    and let $(u,B)$ be the associated global solution to \eqref{eq:non} with $|\Omega| \geq \Omega_0$, constructed in Lemma \ref{lemm:GWP}.
    Then, for every $2 \leq p < \infty$ and $ \ell \in \N \cup \{0\}$, there exists a positive constant $K_{p,\ell}=K_{p,\ell}(\n{u_0}_{L^1}, \n{B_0}_{L^1}
 \n{u_0}_{\dH^{\frac{1}{2}}},  \n{B_0}_{\dH^{\frac{1}{2}}})$, independent of $\Om$, such that 
    \begin{align}
        \n{\nabla^{\ell}\mathcal{N}_{\Om}^{\rm vel}[u](t)}_{L^p}
        &
        \leq 
        K_{p,\ell}
        t^{-\frac{3}{2}(1-\frac{1}{p})-\frac{\ell+1}{2}}
        \mathcal{D}_p(|\Om| t), \label{Prop:dec-p}
        \\
        \n{\nabla^{\ell}\mathcal{N}_{\Om}^{\rm vel}[u](t)}_{L^{4,\infty}}
        &
        \leq 
        K_{4,\ell}
        t^{-\frac{3}{2}(1-\frac{1}{4})-\frac{\ell+1}{2}}
        (1+|\Om|t)^{-\frac{3}{4}} \label{Prop:dec-p-44}
    \end{align}
    for all $t>0$.
    Moreover, for the case of $p=\infty$, there exist a positive constant $K_{\infty,\ell}=K_{\infty,\ell}(\n{u_0}_{L^1}, \n{B_0}_{L^1}
 \n{u_0}_{\dH^{\frac{1}{2}}},  \n{B_0}_{\dH^{\frac{1}{2}}})$ and an absolute positive constant $C$ such that
    \begin{align}
        \n{\nabla^{\ell} \mathcal{N}_{\Om}^{\rm vel}[u](t)}_{L^\infty}
        \leq {}&
        K_{\infty,\ell}
        t^{-\frac{3}{2}-\frac{\ell+1}{2}}
        \mathcal{D}_\infty(|\Om| t) 
        \\
        &
        +
        Ct^{-\frac{5}{2}-\frac{\ell}{2}}
        \mathcal{D}_4(|\Om| t)^2
        \n{u}_{X_{\Om,4}^{0}(t)}
        \n{u}_{X_{\Om,4}^{\ell+1}(t)}^{\frac{1}{2}}
        \n{u}_{X_{\Om,4}^{\ell+2}(t)}^{\frac{1}{2}}
    \end{align}
    for all $t>0$, where 
    \begin{align}
        \n{u}_{X_{\Om,p}^{\ell}(t)}
        :=
        \sup_{\frac{t}{2} \leq \tau \leq t}
        \tau^{\frac{3}{2}(1-\frac{1}{p})+\frac{\ell}{2}}
        \mathcal{D}_p(|\Om|\tau)^{-1}\n{\nabla^{\ell} u(\tau)}_{L^p}.
    \end{align}
\end{lemm}

Hence, to obtain the decay estimate of $\n{\nabla^{\ell}u(t)}_{L^p}$, we only need to focus on the nonlinear term $\mathcal{N}_{\Om}^{\rm vel}[B]$. 
\begin{lemm}\label{prop:nonl-dec-2}
Let the assumptions of Lemma \ref{prop:nonl-dec} be satisfied. Then, for any $2 \leq p < \infty$ and $ \ell \in \N \cup \{0\}$, there exists a positive constant $K_{p,\ell}=K_{p,\ell}(\n{u_0}_{L^1}, \n{B_0}_{L^1}
 \n{u_0}_{\dH^{\frac{1}{2}}},  \n{B_0}_{\dH^{\frac{1}{2}}})$, independent of $\Om$, such that     
 \begin{align}
        \n{\nabla^{\ell}\mathcal{N}_{\Om}^{\rm vel}[B](t)}_{L^p}
        &
        \leq 
        K_{p,\ell}
        t^{-\frac{3}{2}(1-\frac{1}{p})-\frac{\ell+1}{2}}
        \mathcal{D}_p(|\Om| t), \label{Prop:dec-p-B}
        \\
        \n{\nabla^{\ell}\mathcal{N}_{\Om}^{\rm vel}[B](t)}_{L^{4,\infty}}
        &
        \leq 
        K_{4,\ell}
        t^{-\frac{3}{2}(1-\frac{1}{4})-\frac{\ell+1}{2}}
        (1+|\Om|t)^{-\frac{3}{4}} \label{Prop:dec-p-44-B}
    \end{align}
    for all $t>0$. Moreover, for the case of $p=\infty$, there exist a positive constant $K_{\infty,\ell}=K_{\infty,\ell}(\n{u_0}_{L^1}, \n{B_0}_{L^1}
 \n{u_0}_{\dH^{\frac{1}{2}}},  \n{B_0}_{\dH^{\frac{1}{2}}})$ and a positive constant \\$C_{\Omega,\ell}=C_{\Omega,\ell} (\Omega, \n{u_0}_{L^1}, \n{B_0}_{L^1}
 \n{u_0}_{\dH^{\frac{1}{2}}},  \n{B_0}_{\dH^{\frac{1}{2}}}) $ such that
    \begin{align}\label{L-infty-B}
        \n{\nabla^{\ell}\mathcal{N}_{\Om}^{\rm vel}[B](t)}_{L^\infty}
        \leq {}&
        K_{\infty,\ell}
        t^{-\frac{3}{2}-\frac{\ell+1}{2}}
        \mathcal{D}_\infty(|\Om| t) 
        +C_{\Omega,\ell}  t^{-\frac{3}{2}-\frac{\ell+1}{2}}
        \mathcal{D}_\infty(|\Om| t) 
    \end{align}
    for all $t>0$.
\end{lemm}
\begin{proof}
We decompose $\mathcal{N}_{\Om}^{\rm vel}[B]$ into two parts as 
\begin{align}
\mathcal{N}_{\Om}^{\rm vel}[B](t)=
\mathcal{N}_{\Om}^{{\rm vel},1}[B](t)
+
\mathcal{N}_{\Om}^{{\rm vel},2}[B](t), 
\end{align}
with 
\begin{align}
&\mathcal{N}_{\Om}^{{\rm vel},1}[B](t):=
\int_0^{\frac{t}{2}} T_{\Omega} (t-\tau) \mathbb{P} \div (B \otimes B)(\tau)\dtau, \\
&\mathcal{N}_{\Om}^{{\rm vel},2}[B](t):=
\int_{\frac{t}{2}}^{t} T_{\Omega} (t-\tau) \mathbb{P} \div (B \otimes B)(\tau)\dtau. 
\end{align}
It follows from Lemma \ref{thm:lin-decay} and Lemma \ref{lemm:L2-decay} that for any $2 \leq p \leq \infty$
    \begin{align}
        \n{\nabla^{\ell} \mathcal{N}_{\Om}^{{\rm vel},1}[B](t)}_{L^{p}}
        &
        \leq 
        C\int_0^{\frac{t}{2}} 
        (t-\tau)^{-\frac{3}{2}(1-\frac{1}{p})-\frac{\ell+1}{2}}\mathcal{D}_p(|\Om|(t-\tau))
        \n{B(\tau) \otimes B(\tau)}_{L^1}\dtau
        \\
        &
        \leq 
        C
        t^{-\frac{3}{2}(1-\frac{1}{p})-\frac{\ell+1}{2}}\mathcal{D}_p(|\Om| t)
        \int_0^\infty \n{B(\tau)}_{L^2}^2 \dtau \label{dec-L^p-1}
        \\
        &
        \leq 
        C
        t^{-\frac{3}{2}(1-\frac{1}{p})-\frac{\ell+1}{2}}\mathcal{D}_p(|\Om| t)
        \int_0^\infty (1+\tau)^{-\frac{3}{2}} \dtau
        \\
        &
        =
        C
        t^{-\frac{3}{2}(1-\frac{1}{p})-\frac{\ell+1}{2}}\mathcal{D}_p(|\Om| t).
    \end{align}
Next, we consider the $L^p$ estimate for $\mathcal{N}_{\Om}^{{\rm vel},2}[B]$ with $2\leq p<\infty$. By Remark \ref{Rem}, Lemma \ref{lemm:L2-decay}, Lemma \ref{lemm:elem}, the product estimate from Lemma \ref{lemm:prod-est} and the interpolation inequality from Lemma \ref{lemm:interp}, we deduce 
\begin{align}
        \n{
        \nabla^{\ell} 
        \mathcal{N}_{\Om}^{{\rm vel},2}[B](t) 
        }_{L^p}
        &
        \leq
        C\int_\frac{t}{2}^t 
        \mathcal{D}_p(\Omega(t-\tau))
        \n{B(\tau) \otimes B(\tau)}_{\dB_{1,1}^{\ell+1+3(1-\frac{1}{p})}} 
        \dtau
        \\
        &
        \leq
        C\int_\frac{t}{2}^t 
        \mathcal{D}_p(\Omega(t-\tau))
        \n{B(\tau)}_{L^2}
        \n{B(\tau)}_{\dot B_{2,1}^{\ell+1+3(1-\frac{1}{p})}} \dtau
        \\ 
        &
        \leq
        C\int_\frac{t}{2}^t 
        \mathcal{D}_p(\Omega(t-\tau))
        \n{B(\tau)}_{L^2}
        \n{\nabla^{\ell+1} B(\tau)}_{L^2}^{\frac{1}{p}}
        \n{\nabla^{\ell+4} B(\tau)}_{L^2}^{1-\frac{1}{p}}
       \dtau 
        \\
        & 
        \leq
        C
        \int_\frac{t}{2}^t 
        \mathcal{D}_p(\Omega(t-\tau))
        (1+\tau)^{-\frac{3}{2}}\tau^{-\frac{ \ell + 1}{2}-\frac{3}{2}(1-\frac{1}{p})}\dtau  \label{dec-L^p-2}
        \\
        &
        \leq
        C
        (1+t)^{-\frac{3}{2}}
        t^{-\frac{\ell+1}{2}-\frac{3}{2}(1-\frac{1}{p})}
        \int_0^\frac{t}{2}
        \mathcal{D}_p(|\Omega|s)
        \, {\rm d} s 
        \\
        &
        \leq
        C
        (1+t)^{-\frac{3}{2}}
        t^{-\frac{\ell-1}{2}-\frac{3}{2}(1-\frac{1}{p})}
        \mathcal{D}_p(|\Omega| t).
    \end{align}
Based on \eqref{dec-L^p-1} and \eqref{dec-L^p-2}, we finish the proof of \eqref{Prop:dec-p-B} for $2\leq p<\infty$. We notice that the same argument as \eqref{dec-L^p-2}, upon using Remark \ref{Rem}, gives
\begin{align}
  \n{
        \nabla^{\ell} 
        \mathcal{N}_{\Om}^{{\rm vel},2}[B](t) 
        }_{L^{4,\infty}}   
        \leq 
         C
        (1+t)^{-\frac{3}{2}}
        t^{-\frac{\ell-1}{2}-\frac{3}{2}(1-\frac{1}{4})}
    (1+|\Om|t)^{-\frac{3}{4}} . 
\end{align}
On the other hand, by invoking Remark \ref{Rem} again, the argument in \eqref{dec-L^p-1} readily gives
\begin{align}
  \n{
        \nabla^{\ell} 
        \mathcal{N}_{\Om}^{{\rm vel},1}[B](t) 
        }_{L^{4,\infty}}   
        \leq 
         C
        t^{-\frac{3}{2}(1-\frac{1}{4})-\frac{\ell+1}{2}}
    (1+|\Om|t)^{-\frac{3}{4}} . 
\end{align}
Combining the two estimates above, we conclude \eqref{Prop:dec-p-44-B}. 

To proceed, we handle the $L^{\infty}$ estimate for $\mathcal{N}_{\Om}^{{\rm vel},2}[B]$. Similar to \eqref{dec-L^p-2}, we have
\begin{align}
        \n{
        \nabla^{\ell} 
        \mathcal{N}_{\Om}^{{\rm vel},2}[B](t) 
        }_{L^{\infty}}
        &
        \leq
        C\int_\frac{t}{2}^t 
        \mathcal{D}_{\infty}(\Omega(t-\tau))
        \n{B(\tau) \otimes B(\tau)}_{\dB_{1,1}^{\ell+4}} 
        \dtau
        \\
        &
        \leq
        C\int_\frac{t}{2}^t 
        \mathcal{D}_{\infty}(\Omega(t-\tau))
        \n{B(\tau)}_{L^2}
        \n{B(\tau)}_{\dot B_{2,1}^{\ell+4}} \dtau
        \\ 
        &
        \leq
        C\int_\frac{t}{2}^t 
        \mathcal{D}_p(\Omega(t-\tau))
        \n{B(\tau)}_{L^2}
        \n{\nabla^{\ell+3} B(\tau)}_{L^2}^{\frac{1}{2}}
        \n{\nabla^{\ell+5} B(\tau)}_{L^2}^{\frac{1}{2}}
       \dtau 
        \\
        & 
        \leq
        C
        \int_\frac{t}{2}^t 
        \mathcal{D}_{\infty}(\Omega(t-\tau))
        (1+\tau)^{-\frac{3}{2}}\tau^{ -\frac{ \ell + 4}{2} }\dtau 
        \label{dec-L^p-200}    
        \\
        &
        \leq
        C
        (1+t)^{-\frac{3}{2}}
        t^{-\frac{\ell+4}{2}}
        \int_0^\frac{t}{2}
        \mathcal{D}_{\infty} (|\Omega|s)
        \, {\rm d} s 
        \\
        &
        \leq 
        C
        (1+t)^{-\frac{3}{2}}
        t^{-\frac{\ell+4}{2}}
        |\Omega|^{-1}    \log(e+|\Omega|t) \\
        &
        \leq
        C_{\Omega,\ell} 
        (1+|\Omega|t)^{-1}  
        t^{-\frac{\ell+ 4}{2}}
      ,
    \end{align}
where we used the fact that
\begin{align}
 (1+t)^{-\frac{3}{2}} |\Omega|^{-1} 
 (1+|\Omega|t)  \log(e+|\Omega|t) 
 \leq C_{\Omega}
\end{align}
for all $|\Omega| \geq \Omega_0$ and $t>0$. 
\end{proof}

We are in a position to prove the nonlinear decay estimates of the velocity field. By Lemma \ref{thm:lin-decay}, Lemma \ref{prop:nonl-dec} and Lemma \ref{prop:nonl-dec-2}, we see for $2 \leq p < \infty$ that 
\begin{align}
    \n{\nabla^{\ell}  u(t)}_{L^p} 
    \leq{}& 
    \n{ \nabla^{\ell} T_\Omega(t)u_0}_{L^p}
    +
    \n{\nabla^{\ell} \mathcal{N}_{\Om}^{\rm vel}[u](t)}_{L^p} 
    +
     \n{\nabla^{\ell} \mathcal{N}_{\Om}^{\rm vel}[B](t)}_{L^p}
    \\
    \leq{}& 
    \n{ \nabla^{\ell} T_\Omega(t)u_0}_{L^p}
    +
    K_{p,\ell}
    t^{-\frac{3}{2}(1-\frac{1}{p})-\frac{\ell+1}{2}}
    \mathcal{D}_p(|\Om| t) \label{nonlin:p-small}\\
    ={}& 
    o\sp{t^{-\frac{3}{2}(1-\frac{1}{p})-\frac{\ell}{2}}
    \mathcal{D}_p(|\Om| t)}
\end{align}
as $t \to \infty$. For the case of $L^\infty$, we infer by the same token that
\begin{align}
    \n{\nabla^{\ell}  u(t)}_{L^\infty} 
    \leq{}& 
     \n{ \nabla^{\ell} T_\Omega(t)u_0}_{L^\infty} 
     +
     \n{\nabla^{\ell} \mathcal{N}_{\Om}^{\rm vel}[u](t)}_{L^\infty} 
    +
     \n{\nabla^{\ell} \mathcal{N}_{\Om}^{\rm vel}[B](t)}_{L^\infty} \\
    \leq{}& 
    \n{ \nabla^{\ell} \mathcal{T}_\Omega(t)u_0}_{L^\infty}
    + 
    K_{\infty,\ell}
    t^{-\frac{3}{2}-\frac{\ell+1}{2}}
    \mathcal{D}_\infty(|\Om| t) \\
    &
    +
    Ct^{-\frac{5}{2}-\frac{\ell}{2}}
    \mathcal{D}_4(|\Om| t)^2
    \n{u}_{X_{\Om,4}^{0}(t)}
    \n{u}_{X_{\Om,4}^{\ell+1}(t)}^{\frac{1}{2}}
    \n{u}_{X_{\Om,4}^{\ell+2}(t)}^{\frac{1}{2}}
    \\
    & +C_{\Omega,\ell}  t^{-\frac{3}{2}-\frac{\ell+1}{2}}
        \mathcal{D}_\infty(|\Om| t) 
        \\ 
    ={}& 
    o \sp{t^{-\frac{3}{2}-\frac{\ell}{2}}
    \mathcal{D}_\infty(|\Om| t)}
\end{align}
as $t \to \infty$,
where we have used 
\begin{align}
\n{u}_{X_{\Om,4}^{0}(t)}
\n{u}_{X_{\Om,4}^{\ell+1}(t)}^{\frac{1}{2}}
\n{u}_{X_{\Om,4}^{\ell+2}(t)}^{\frac{1}{2}}=o(1), \quad t \to \infty. 
\end{align}
Observe that the above fact comes from \eqref{nonlin:p-small} with $p=4$. This verifies the decay estimates of the velocity in Theorem \ref{thm-1}.

Next, we consider the enhanced decay estimate of the velocity field, by assuming moreover that $|x|u_0(x), |x|B_0(x) \in L^1(\R^3)$. For $2\leq p<\infty$, we deduce from Lemma \ref{thm:lin-decay}, Lemma \ref{prop:nonl-dec} and \eqref{Prop:dec-p-B} that
\begin{align}
    \n{\nabla^\ell  u(t)}_{L^p} 
    \leq{}& 
     \n{ \nabla^{\ell} T_\Omega(t)u_0}_{L^p}
    +
    \n{\nabla^{\ell} \mathcal{N}_{\Om}^{\rm vel}[u](t)}_{L^p} 
    +
     \n{\nabla^{\ell} \mathcal{N}_{\Om}^{\rm vel}[B](t)}_{L^p} \\
    \leq{}& 
    Ct^{-\frac{3}{2}(1-\frac{1}{p})-\frac{\ell+1}{2}}
    \mathcal{D}_p(|\Omega| t)\n{|x|u_0(x)}_{L^1(\R^3_x)} \label{est:nonlin-Lp-100}
    \\
    &
    + 
     K_{p,\ell}
    t^{-\frac{3}{2}(1-\frac{1}{p})-\frac{\ell+1}{2}}
    \mathcal{D}_p(|\Om| t) \\
    \leq{}& 
    C t^{-\frac{3}{2}(1-\frac{1}{p})-\frac{\ell+1}{2}}
    \mathcal{D}_p(|\Om| t)
\end{align}
for all $t>0$. For the case of $L^\infty$, we deduce from Lemma \ref{thm:lin-decay}, Lemma \ref{prop:nonl-dec} and \eqref{L-infty-B} that 
\begin{align}
    \n{\nabla^\ell   u(t)}_{L^\infty} 
    \leq{}& 
    \n{ \nabla^\ell  T_\Omega(t)u_0}_{L^\infty}
    +
    \n{\nabla^\ell  \mathcal{N}^{\rm vel}_\Om[u](t)}_{L^\infty} 
     +
     \n{\nabla^{\ell} \mathcal{N}_{\Om}^{\rm vel}[B](t)}_{L^\infty}
     \\
    \leq{}& 
    Ct^{-\frac{3}{2}-\frac{\ell +1}{2}}
    \mathcal{D}_\infty(|\Omega| t)\n{|x|u_0(x)}_{L^1(\R^3_x)} 
    \\
    &
    + 
    K_{\infty,\ell }
    t^{-\frac{3}{2}-\frac{\ell +1}{2}}
    \mathcal{D}_\infty(|\Om| t) \label{est:nonlin-Linf}\\
    &
    +
    Ct^{-\frac{5}{2}-\frac{\ell }{2}}
    \mathcal{D}_4(|\Om| t)^2
    \n{u}_{X_{\Om,4}^{0}(t)}
    \n{u}_{X_{\Om,4}^{\ell +1}(t)}^{\frac{1}{2}}
    \n{u}_{X_{\Om,4}^{\ell +2}(t)}^{\frac{1}{2}}
    \\
    &
      +C_{\Omega,\ell}  t^{-\frac{3}{2}-\frac{\ell+1}{2}}
        \mathcal{D}_\infty(|\Om| t) 
        \\
    \leq{}& 
    C_{\Omega,\ell} t^{-\frac{3}{2}-\frac{\ell +1}{2}}
    \mathcal{D}_\infty(|\Om| t)
\end{align}
for all $t \geq 1$. Here, we have invoked 
\begin{align}
    \sup_{t>0}
    \sp{
    \n{u}_{X_{\Om,4}^{0}(t)}
    \n{u}_{X_{\Om,4}^{\ell+1}(t)}^{\frac{1}{2}}
    \n{u}_{X_{\Om,4}^{\ell+2}(t)}^{\frac{1}{2}}
    } 
    \leq C,
\end{align}
which is ensured by \eqref{est:nonlin-Lp-100}. This proves the enhanced decay estimate of velocity in Theorem \ref{thm-2}. Remark that the logarithmic correction for $p=4$ in $L^{4,\infty}$ is proved similarly, the details of which are omitted.

\subsection{Decay estimates of the magnetic field} Now we turn to the decay estimates of the magnetic field, based on the decay estimates of the velocity field at hand. We only focus on the enhanced decay estimate in Theorem \ref{thm-2}, since the corresponding estimate in Theorem \ref{thm-1} follows in the same way. To this end, we decompose $\mathcal{N}^{\rm mag}[u,B]$ into two parts as 
\begin{align}
\mathcal{N}^{\rm mag}[u,B](t)=
\mathcal{N}^{{\rm mag},1}[u,B](t)
+
\mathcal{N}_{\Om}^{{\rm mag},2}[u,B](t), 
\end{align}
with 
\begin{align}
&\mathcal{N}^{{\rm mag},1}[u,B](t):=
\int_0^{\frac{t}{2}}e^{(t-\tau) \Delta} \nabla \times (B \times u) (\tau) 
 \dtau, \\
&\mathcal{N}^{{\rm mag},2}[u,B](t):=
\int_{\frac{t}{2}}^{t} e^{(t-\tau) \Delta} \nabla \times (B \times u) (\tau) 
 \dtau. 
\end{align}
\begin{lemm}\label{lemm:Duhamel-B}
Let the assumptions of Theorem \ref{thm-2} be satisfied. Then, for any $2 \leq p \leq  \infty$ and $ \ell \in \N \cup \{0\}$, there exists a positive constant $K_{p,\ell}=K_{p,\ell}(u_0,B_0)$, independent of $\Om$, such that    
\begin{align}
\n{\nabla^{\ell} \mathcal{N}^{{\rm mag},1}[u,B](t)}_{L^p}
        \leq 
        K_{p,\ell}
        t^{-\frac{3}{2}(1-\frac{1}{p})-\frac{\ell+1}{2}}
\end{align}
 for all $t>0$. For any $2 \leq p < \infty$ and $ \ell \in \N \cup \{0\}$, there exists a positive constant $K_{p,\ell}=K_{p,\ell}(u_0,B_0)$, independent of $\Om$, such that    
\begin{align}\label{B-p}
\n{\nabla^{\ell} \mathcal{N}^{{\rm mag},2}[u,B](t)}_{L^p}
        \leq 
        K_{p,\ell}
        t^{-\frac{3}{2}(1-\frac{1}{p})-\frac{\ell+3}{2}}
\end{align}
for all $t>0$. For any $ p =\infty$ and $ \ell \in \N \cup \{0\}$, there exists a positive constant $C_{\Omega,\ell}=C_{\Omega,\ell}(u_0,B_0)$ such that   
\begin{align}\label{B-infty}
\n{\nabla^{\ell} \mathcal{N}^{{\rm mag},2}[u,B](t)}_{L^p}
        \leq 
       C_{\Omega,\ell}
        t^{-\frac{3}{2} -\frac{\ell+3}{2}}
\end{align}
for all $t>1$. 
\end{lemm}
\begin{proof}

For any $2 \leq p \leq  \infty$, it follows that
 \begin{align}
        \n{\nabla^{\ell} \mathcal{N}^{{\rm mag},1}[u,B](t)}_{L^{p}}
        &
        \leq 
        C\int_0^{\frac{t}{2}} 
        (t-\tau)^{-\frac{3}{2}(1-\frac{1}{p})-\frac{\ell+1}{2}}
        \n{B(\tau) \times u(\tau)}_{L^1}\dtau
        \\
        &
        \leq 
        C
        t^{-\frac{3}{2}(1-\frac{1}{p})-\frac{\ell+1}{2}}
        \int_0^\infty \n{B(\tau)}_{L^2}\n{u(\tau)}_{L^2} \dtau \label{dec-L^p-100}
        \\
        &
        \leq 
        C
        t^{-\frac{3}{2}(1-\frac{1}{p})-\frac{\ell+1}{2}}
        \int_0^\infty (1+\tau)^{-\frac{3}{2}} \dtau
        \\
        &
        =
        C
        t^{-\frac{3}{2}(1-\frac{1}{p})-\frac{\ell+1}{2}}.
    \end{align}
For any $2\leq p<\infty$, we infer from Lemma \ref{lemm:L2-decay} and \eqref{est:nonlin-Lp-100} that 
\begin{align}
  & \n{\nabla^{\ell} \mathcal{N}^{{\rm mag},2}[u,B](t)}_{L^{p}} \\
        & \quad 
        \leq    
        C \int_{\frac{t}{2}}^{t}
        \n{G(t-\tau)}_{L^2}  
        \n{ \nabla^{\ell+1} (B\times u)  }_{L^{\frac{2p}{2+p}  }}  \dtau \\
        & \quad   \leq    
        C \int_{\frac{t}{2}}^{t}
        (t-\tau)^{ -\frac{3}{4} } 
        \sp{
        \n{\nabla^{\ell+1}B  }_{L^2} \n{u}_{L^p}
        +\n{\nabla^{\ell}B  }_{L^2} \n{\nabla u}_{L^p}
        +...+ \n{B}_{L^2} \n{\nabla^{\ell+1}u  }_{L^p} 
        }
         \dtau \\
        &  \quad  \leq    
        C \int_{\frac{t}{2}}^{t}
         (t-\tau)^{ -\frac{3}{4} }
         \tau^{  -\frac{3}{4}- \frac{3}{2}(1-\frac{1}{p} )- \frac{\ell+2}{2}   } \mathcal{D}_{p}(|\Omega|\tau)
          \dtau \\
        &  \quad  \leq  
        C  t^{  -\frac{3}{2}(1-\frac{1}{p} )- \frac{\ell+3}{2}   } 
\end{align}
for all $t>0$. Here, we denote by $G(t)$ the standard $3$D Gaussian. This proves \eqref{B-p}. The proof of \eqref{B-infty} follows similarly by using \eqref{est:nonlin-Linf}. 
\end{proof}

With Lemma \ref{lemm:Duhamel-B} at hand, we show the decay estimate of the velocity field. By Lemma \ref{thm:heat-decay} and \eqref{B-p}, we see for $2\leq p<\infty$ that 
\begin{align}
\n{\nabla^{\ell} B(t)}_{L^p} 
& \leq  
\n{\nabla^{\ell}e^{t \Delta}B_0}_{L^p}
+
\n{\nabla^{\ell}e^{t \Delta}  \mathcal{N}^{{\rm mag},1}[u,B](t)    }_{L^p}
+
\n{\nabla^{\ell}e^{t \Delta}  \mathcal{N}^{{\rm mag},2}[u,B](t)    }_{L^p} \\
& \leq 
Ct^{-\frac{3}{2}(1-\frac{1}{p})-\frac{\ell+1}{2}}
+
 K_{p,\ell}
        t^{-\frac{3}{2}(1-\frac{1}{p})-\frac{\ell+1}{2}}
        +
         K_{p,\ell}
        t^{-\frac{3}{2}(1-\frac{1}{p})-\frac{\ell+3}{2}} \\
        & \leq 
        Ct^{-\frac{3}{2}(1-\frac{1}{p})-\frac{\ell+1}{2}}
\end{align}
for all $t >1$. The case of $p=\infty$ follows similarly with the help of \eqref{B-infty}.

\vspace{10mm}

\noindent{{\bf{Acknowledgements}}} 
The work of Mikihiro Fujii was supported by JSPS KAKENHI, Grant Number JP25K17279. 
The work of Yang Li was supported by National Natural Science Foundation of China, Grant number 12571228 and Natural Science Foundation of Anhui Province, Grant number 2408085MA018.  
\vspace{4mm}

\noindent{{\bf{Data Availability}}} Data sharing is not applicable to this article as no datasets were generated or analysed
during the current study.

\vspace{4mm}

\noindent{{\bf{Conflicts of interest}}} All authors certify that there are no conflicts of interest for this work.

\end{document}